\documentclass[12pt]{article}

\usepackage{amsfonts}
\usepackage{amsmath}
\usepackage{amssymb}
\usepackage{amsthm}
\usepackage{amscd}
\usepackage{xcolor}

\theoremstyle{plain} 
\newtheorem{theo}{Theorem}[section]
\newtheorem{prop}[theo]{Proposition}
\newtheorem{lem}[theo]{Lemma}

\newtheorem{coro}[theo]{Corollary}
\theoremstyle{definition}
\newtheorem{defi}[theo]{Definition}
\newtheorem{rk}[theo]{Remark}

\newtheorem{exa}[theo]{Example}

\numberwithin{equation}{section}

\newcommand{\from}{\colon}
\newcommand{\N}{{\mathbb N}}
\newcommand{\Z}{{\mathbb Z}}
\newcommand{\R}{{\mathbb R}}

\newcommand{\calE}{{\cal{E}}}

\newcommand{\ccalE}{\check{\cal{E}}}

\newcommand{\calD}{{\cal{D}}}
\newcommand{\calH}{{\cal{H}}}

\newcommand{\calEb}{\calE^{(\beta)}}
\newcommand{\calEt}{\calE^{(t)}}

\newcommand{\form}{a}

\newcommand{\ccalD}{\check{\cal{D}}}
\DeclareMathOperator{\ran}{ran}

\newcommand{\dom}{\mathrm{dom}}

\newcommand{\trans}{\mathrm{trans}}
\newcommand{\rec}{\mathrm{rec}}
\newcommand{\cons}{\mathrm{cons}}
\newcommand{\diss}{\mathrm{diss}}

\begin{document}
\title{Decomposition formulae for Dirichlet forms and their corollaries}
\author{\normalsize
Ali BenAmor\footnote{corresponding author}\footnote{High institute for transport and logistics. University of Sousse, Tunisia. E-mail: ali.benamor@ipeit.rnu.tn},
Rafed Moussa \footnote{Department of Mathematics, High school of sciences and technology of Hammam Sousse. University of Sousse, Tunisia. E-mail: rafed.moussa@gmail.com}
}

\date{ }

\maketitle

\begin{abstract}
We provide decompositions of Dirichlet forms into recurrent and transient parts as well as into conservative and dissipative parts, in the framework of Hausdorff state spaces. Combining both formulae we write every Dirichlet form as the sum of a recurrent, dissipative and transient conservative Dirichlet forms. Besides, we prove that Mosco convergence preserves invariant sets and that a Dirichlet form shares the same invariants sets with its approximating Dirichlet forms $\calEt$  and $\calEb$. Finally we show the  equivalence between conservativeness (resp. dissipativity) of a Dirichlet form and the conservativeness (reps. dissipativity) of $\calEt$  and $\calEb$.  The elaborated results are enlightened by some examples.
\end{abstract}

\textbf{MSC 2010:} 47A07, 46C05, 46C07, 47B25, 46E30.

\textbf{Keywords:} Dirichlet forms, invariant sets, conservative, dissipative, Mosco convergence

\section{Introduction}

Among interesting global properties for Dirichlet forms, we list recurrence, transience, conservativeness and dissipativity. Unfortunately, a Dirichlet form $\calE$ may fail to have any of the mentioned properties. To overcome this problem we shall establish  decompositions formulae for $\calE$ into the sum of a recurrent and a transient Dirichlet form as well as into the sum of a conservative and a dissipative Dirichlet forms. Then combining both formulae we shall write any Dirichlet form as the sum of a recurrent, dissipative and transient-conservative forms.\\
The motivation rests on the existence of Dirichlet forms which are neither recurrent nor transient or neither conservative nor dissipative. Moreover there are Dirichlet forms which are simultaneously transient and conservative. Hence the second decomposition is finer than the first one. Besides, the mentioned decompositions lead to ergodic decompositions of a given Dirichlet form. Also known facts about recurrence, transience, conservativeness  or dissipativity can be used to investigate properties of the considered Dirichlet forms, by means of investigations of each part separately.\\
Let us quote that decomposition into recurrent and transient parts already exist in the literature, mainly in \cite{Meyer-Ergo-Lim,Fukushima-Ergodic,Dynkin-Excessive,Fitzsimmons-Maisonneuve} (in implicit form) and explicitly in \cite[Theorem 1.3]{Kuwae},  for non-symmetric quasi regular (semi)-Dirichlet forms on locally compact metric state spaces. The most general and purely analytic framework can be found in \cite{Meyer-Ergo-Lim}.\\
At this stage we draw the attention of the reader to the following connotations. In the above-mentioned references, the authors use the terminology 'conservative' for what we call and is in fact 'recurrent' and 'dissipative' for what we call and is in fact 'transient'.\\
In this respect our major contributions are, among others, first to establish the decomposition of a symmetric Dirichlet forms  into conservative and dissipative parts. Second, provide  decomposition of a symmetric Dirichlet form into recurrent plus dissipative plus transient-conservative parts, in the general framework of Hausdorff topological spaces. This leads, in particular to the fact that every dissipative form is  transient, whereas the converse is not true in general. Furthermore we shall precise under which conditions the considered parts of a given Dirichlet form coincide with each other. As a byproduct one obtains criteria for conservativeness and dissipation. Pushing our analysis forward, we shall exploit  the established decompositions to study relationship between conservativeness, respectively,  dissipativity  of a Dirichlet form and its Deny--Yoshida or time dependent approximating forms.\\
As invariant sets emerge naturally in our framework, we shall also investigate some of their properties. As a new result, we shall prove the remarkable  fact  that Mosco convergence preserves invariant sets. This shows, in particular, that the limit Dirichlet form has much more invariant sets than its approximating sequence. Moreover we demonstrate that parts of Dirichlet forms on invariant sets inherit Mosco convergence.\\
Finally combining all these results, will lead to the fact that the conservative part, respectively the dissipative part, of the Deny--Yoshida or of the time dependent approximation of $\calE$ converges to the corresponding part of $\calE$.\\
The paper is organized  as follows: In section 2, we present the framework and recall some known results. As invariant sets play an important role for our method we will revisit them in the third section. Section 4 is devoted to establish the decompositions formulae and their consequences. Some illustrating examples are given in section 5
\section{The framework and preparing results}

Let $X$ be a Hausdorff topological space and $m$ be a positive  $\sigma$-finite Radon measure on some $\sigma$-algebra $\mathcal{A}$ of subsets of $X$, with full support $X$. For every $p\in [1,\infty]$, the symbol $L^p$ stands for the usual Lebesgue space $L^p(X,m)$. We shall denote by $(\cdot,\cdot)$ the scalar product on $L^2$.\\
Unless otherwise stated, all equalities and inequalities considered in the paper has to be understood  in the sense $m$-a.e.\\
A Dirichlet form  $\calE$ with domain $\calD\subset L^2$ is a densely defined closed quadratic form such that
\begin{align*}
u\in\calD\Rightarrow\,u_{0,1}:=\min(\max(u,0),1)\in\calD\ \text{ and } \calE[u_{0,1}]\leq\calE[u].
\end{align*}
We draw the attention of the reader that we shall consider only symmetric Dirichlet forms.\\
For every $\beta>0$ we set $\calEb$ the Deny--Yosida approximation of $\calE$:
\begin{eqnarray}
\dom(\calEb)=L^2,\ \calEb[u]:= \beta(u-\beta K_\beta u,u),\ \forall\,u\in L^2.
\end{eqnarray}
Then $(\calEb)_{\beta>0}$ are bounded monotone increasing Dirichlet forms. By the spectral theorem we get the following
\begin{eqnarray}
\calD=\{u\in L^2:\ \lim_{\beta\uparrow\infty}\calEb[u]<\infty\},\ \calE[u]=\lim_{\beta\uparrow\infty}\calEb[u],\ \forall\,u\in\calD.
\end{eqnarray}
It is well known that $\calEb$ converges in the strong resolvent sense, and hence in the sense of Mosco to $\calE$ as $\beta\uparrow\infty$.\\
Let $(T_t)_{t>0}$ be the semigroup family related to $\calE$. It is well known that $(T_t)_{t>0}$ is a strongly continuous family of Markovian selfadjoint operators.
For every $t>0$ we set $\calEt$ the 'time dependent' approximation of $\calE$:
\begin{eqnarray}
\dom(\calEt)=L^2,\ \calEt[u] = \frac{1}{t}(u- T_t u,u),\ \forall\,u\in L^2.
\end{eqnarray}
Then $\calEt$ are bounded monotone decreasing Dirichlet forms. By the spectral theorem, once again, we get the following
\begin{eqnarray}
\calD=\{u\in L^2:\ \lim_{t\downarrow 0}\calEt[u]<\infty\},\ \calE[u]=\lim_{t\downarrow 0}\calEt[u]=\sup_{t>0}\calEt[u],\ \forall\,u\in\calD.
\label{monotoneLim}
\end{eqnarray}
It also holds that $\calE$ is the Mosco limit of $\calEt$ as $t\downarrow 0$.\\
It is also well known that $T_t$ extends to a Markovian semigroup of contractions on $L^\infty$ which we still denote by $T_t,\ t>0$. This extension goes as follows (see \cite[p.6]{Chen-Fukushima}): By the $\sigma$-finiteness of $(X,m)$ there is an increasing  sequence $(\varphi_k)\subset L^1\cap L^\infty$ such that $0<\varphi_k\leq 1$ and $\varphi_k\uparrow 1$. Let $0\leq u\in L^\infty$. Then $u\varphi_k\in L^2$ and $u\varphi_k\uparrow u$. Using monotonicity property for $T_t$,  we define
\[
T_t u:= \lim_{k\to\infty} T_t(u\varphi_k).
\]
For $u\in L^\infty$ of indefinite sign we define $T_tu=T_tu^+ - T_tu^-$.\\
We use the same procedure to extend $K_\beta$ to $L^\infty$.\\
The following is well known and can be found in \cite[Lemma 1.1.6-1.1.7, pp.7-9]{Chen-Fukushima}.
\begin{lem}{(Representation formulae for $\calEb$ and $\calEt$)}.
For every  $u\in L^2\cap L^\infty,\ t,\beta>0$ set
\[
\sigma_t(u)=\frac{1}{t} (T_t u^2 -2uT_tu + u^2T_t1)\ \text{\rm and}\ \kappa_\beta(u)=\beta(K_\beta u^2 -2uK_\beta u + u^2 K_\beta 1).
\]
Then $\sigma_t(u)\geq 0,\ \kappa_\beta(u)\geq 0$ and
\begin{align}
\calEt[u]=\frac{1}{2}\int_X \sigma_t(u)\,dm + \frac{1}{t}\int_X(1-T_t 1)u^2(x)\,dm.
\label{AlternativeRep}
\end{align}
\begin{align}
\calEb[u]=\frac{\beta}{2}\int_X \kappa_\beta(u)\,dm + \beta\int_X(1-\beta K_\beta 1)u^2(x)\,dm.
\end{align}
\label{KernelRep}
\end{lem}
\section{Invariant sets, revisited}
We collect in this section some new and known results related to invariant sets. As aspect of novelty, we shall demonstrate that Mosco convergence preserves invariance of sets as well as  convergence of the traces with respect to invariant sets (convergence of parts of Dirichlet forms on invariant sets).\\
A subset $Y\subset X$ is said to be $T_t$-invariant (or $\calE$-invariant) whenever it is measurable and
\begin{align*}
T_t(1_Y u)= 1_Y T_t u, \text{ for any } u\in L^2 \text{ and some (any) } t>0.
\end{align*}
For short we shall simply say that $Y$  is invariant.\\
Thanks to the relationship between closed quadratic forms their semigroups and their resolvents, invariance of a measurable set $Y$ is equivalent to
\[
K_\alpha(1_Yu)=1_YK_\alpha u,\   \text{ for any } u\in L^2 \text{ and some (any) } \alpha>0.
\]
It is well known that the following two conditions are equivalent to the invariance of a measurable set $Y$
\[
T_t 1_Y=1_YT_t1\ \text{ for some (any) } \,t>0,\ K_\beta 1_Y=1_YK_\beta1\  \text{ for some (any) } \,\beta>0
\]
where, in this context, the semigroup $T_t$ and the resolvent $K_\beta$ are  those  induced by the $L^2$-semigroup and the $L^2$-resolvent of $\calE$ on $L^\infty$.\\
Let us analyze relationship between invariance of a given set  with respect to $\calE$ and with respect to their approximates $\calEt$ and  $\calEb$.
\begin{theo}
Let $Y$ be measurable. Then the following assertions are equivalent
\begin{enumerate}
\item $Y$ is invariant w.r.t. $\calEt$ for every $t>0$.
\item  $Y$ is invariant w.r.t. $\calEb$ for every $\beta>0$.
\item  $Y$ is invariant w.r.t. $\calE$.
\end{enumerate}
\label{ApproxInvariance}
\end{theo}
\begin{proof}
Let $Y$ be measurable. The implications $(a)\to(c)$ and  $(b)\to(c)$ follow from Theorem \ref{InvarianceMosco}-(a).\\
Conversely assume that $Y$ is  $\calE$-invariant. Let $(T_s^{(t)})_{s>0}$ be the semigroup related to $\calEt$.  Making use of the spectral theorem, an elementary computation leads to
\[
T_s^{(t)}=\exp(-s/t(1-T_t))=e^{-s/t}\sum_{k=0}^\infty \frac{1}{k!}(s/t)^k T_t^k, \forall\,s,t>0,
\]
uniformly. Regarding the $\calE$-invariance of $Y$, induction on $k$ leads to
\[
T_t^k(1_Y u)= T_t^{k-1}(1_YT_t u)=1_YT_t^k u,\ \forall\,t>0,\ k\in\N,\ u\in L^2.
\]
Hence  $T_s^{(t)} (1_Yu)= 1_Y T_s^{(t)}$ and $Y$ is $\calEt$ for every $t>0$.\\
To prove the implication  $(c)\to(b)$, let $T_t^{(\beta)}$ be the semigroup related to $\calEb$. If $Y$ is $\calE$-invariant, then induction on $k$ leads to
\[
K_\alpha^k(1_Yu)=1_YK_\alpha^k u,\ \forall\,\alpha>0,\ k\in\N,\ u\in L^2.
\]
Obviously $T_t^{(\beta)}=\exp(t\beta(\beta K_\beta -1))$. Thus writing the latter identity as a series and   using the induction formula we obtain the $\calEb$-invariance of $Y$ and the proof is finished.
\end{proof}
The following lemma is well known (see \cite[Theorem 1.6.1, p.54]{Fukushima}), we include it for the convenience of the reader. At this stage, we recall that a quadratic form with domain in $L^2$ is  a Dirichlet form in the wide sense whenever it fulfills all properties of a Dirichlet form except being densely defined.
\begin{lem}
Let $Y\subset X$ be measurable. Then the following assertions are equivalent.
\begin{enumerate}
\item $Y$ is invariant.
\item For each $u\in\calD$, $1_Yu,1_{Y^c}\in\calD$ and
\[
\calE[u]=\calE[1_Y u] + \calE[1_{Y^c} u].
\]
Moreover if $Y$ is invariant then the quadratic form  $\calE^Y$ defined by
\begin{eqnarray}
\calD^Y:=\dom(\calE^Y)=\calD,\ \calE^Y[u]=\calE[1_Yu],
\end{eqnarray}
is a Dirichlet form in the wide sense in $L^2$ and is in fact a Dirichlet form in $L^2(Y,m|Y)$.
\end{enumerate}
\label{PartOfE}
\end{lem}
\begin{proof}
We include the proof, for the sake of completeness.\\
$(a)\to(b)$: Let $u\in\calD$. Then the measurability of $Y$ yields the measurability of $1_Yu$. Owing to the invariance of $Y$ an easy computation yields
\begin{eqnarray}
\calEt[u] =\calEt[1_Yu] + \calEt[1_{Y^c}u].
\end{eqnarray}
Hence both $\calEt[1_Yu]$ and $\calEt[1_{Y^c}u]$ are dominated by $\calE[u]$, so that $1_Yu,\ 1_{Y^c}u$ lie in $\calD$. Now Letting $t\downarrow 0$ in the latter formula yields (b).\\
Conversely suppose that (ii) holds true. Then for any $u,v\in\calD$ it holds $\calE(1_Yu,v)=\calE(1_Yu,1_Yv)=\calE(u,1_Yv)$. Thus for any $f\in L^2,v\in\calD$ we obtain
\[
\calE_\beta(K_\beta 1_Yf,v)=(1_Yf,v)=(f,1_Yv)=\calE_\beta(K_\beta f,1_Yv)=\calE_\beta(1_YK_\beta f,v),
\]
which amounts to $K_\beta 1_Yf=1_YK_\beta f$. Hence $Y$ is $T_t$-invariant.\\
Let us now prove the last statement. We first prove that $\calE^Y$ is closed in $L^2(Y,m|_Y)$. Let $(u_n)\subset\calD$ be a $\calE^Y_1$-Cauchy sequence. Then the sequence $v_n=1_Eu_n$ is a $\calE_1$-Cauchy sequence (by (c)). Thus there is $v\in\calD$ such that $\calE_1[v_n-v]\to 0$. Obviously $v=0,\ a.e.$ on $Y^c$. Thus  $\calE^Y_1[u_n-v]=\calE_1[v_n-v]\to 0$ and $\calE^Y$ is closed. Finally, for $u\in\calD$ we have $(1_Y u\vee 0)\wedge 1=1_Y(u\vee 0)\wedge 1\in\calD$ and $\calE^Y[(u\vee 0)\wedge 1]=\calE[1_Y(u\vee 0)\wedge 1]= \calE[(1_Yu\vee 0)\wedge 1]\leq\calE[1_Yu]=\calE^Y[u]$. Hence $\calE^Y$ satisfies  Markov property and is therefore  a Dirichlet form. Obviously $\calE^Y$ is densely defined and all these considerations imply that the form is a Dirichlet form.
\end{proof}
For an invariant set $Y$,  we clarify now  the relationship between the part of $\calE$ in $L^2(Y,m|_Y)$, $\calE^Y$ and the trace with respect to the measure $m|_Y$, which we denote by $\check{\calE^Y}$. Let us indicate that the concept of the trace of a Dirichlet form was introduced in \cite[Chap.6.2]{Fukushima}. However we shall adopt the method developed in \cite[Theorem 2.4]{BBST}
\begin{prop}
Let $Y$ be invariant and $J$ the linear operator defined by
\[
J:(\calD,\calE_1)\to L^2(Y,m|_Y),\ Ju = u|_Y.
\]
Then $\check{\calE^Y}=\calE^Y|_{\ran J}$.
\end{prop}
\begin{proof}
We follow the construction made in \cite{BBST}.\\
For each $\lambda>0$, we set $P_\lambda$ the $\calE_\lambda$-orthogonal projection of $\calD$ onto the $\calE_\lambda$-orthogonal of $\ker J$, $\ker J^{\perp_{\calE_\lambda}}$. Let $u\in\calD$. Then $P_\lambda u$ is the unique element form $\calD$ such that $u-P_\lambda u\in\ker J$ and $P_\lambda u\in{\ker J}^{\perp_{\calE_\lambda}}$. Thus $P_\lambda u=u|_Y$ and.
\begin{align}
\ccalE_\lambda[Ju]:=\calE_\lambda[P_\lambda u]=\calE^Y_\lambda [u]\to \calE^Y[u]\ \text{ as } \lambda\downarrow 0.
\end{align}
Since by Lemma \ref{PartOfE} the latter form is closed, the claim follows from \cite[Theorem 2.4]{BBST}.

\end{proof}
Next we discuss the effect of Mosco convergence on invariant sets.\\
Let us recall the definition of Mosco convergence, see \cite[Definition 2.1.1]{Mosco1994}.
Let $(\form_n)$ be a sequence of positive quadratic forms in a Hilbert space $\calH$, $\form_\infty$ a quadratic form in $\calH$.
We say that $(\form_n)$ \emph{Mosco-converges} to the form $\form_\infty$ in $\calH$ provided
\begin{enumerate}
  \item[(M1)]
    for all $(u_n)$ in $\calH$, $u\in \calH$ such that $u_n\to u$ weakly in $\calH$ we have $\liminf_{n\to\infty} \form_n[u_n] \geq \form_\infty[u]$,
  \item[(M2)]
    for all $u\in \calH$ there exists $(u_n)$ in $\calH$ such that $u_n\to u$ in $\calH$ and $\limsup_{n\to\infty} \form_n[u_n]\leq \form_\infty[u]$.
\end{enumerate}
Note that for this definition we extend the quadratic forms to the whole space by setting them $+\infty$ for elements not in their domain.\\
According to \cite[Corollary 2.6.1]{Mosco1994}, Mosco convergence is equivalent to strong resolvent convergence of the corresponding resolvents and hence equivalent to strong convergence of the corresponding semigroups.\\
We also quote the known fact that Mosco limit of a sequence of Dirichlet forms is, in general, a Dirichlet form in the wide sense.
\begin{lem}
Let $P:L^2\to L^2$ be a projection and $(\calE^k)$ be a sequence of Dirichlet forms with domains $\calD_k\subset L^2$. Assume that $\calE^k$ converges in the sense of Mosco to a Dirichlet form $\calE^\infty$ and that $PT_t^{(k)} = T_t^{(k)}P$ for each $k$ and some (and hence every) $t>0$. Then $PT_t^{\infty}=T_t^{\infty}P$ for every $t>0$.
\label{projection}
\end{lem}
\begin{proof}
According to \cite[Corollary 2.6.1]{Mosco1994}, Mosco convergence is equivalent to strong convergence of the related semigroups, a fact from which the result follows.
\end{proof}
\begin{theo}
Let $(\calE^k)$ be a sequence of Dirichlet forms with domains $\dom(\calE^k)\subset L^2$. Assume that $\calE^k$ converges in the sense of Mosco to a Dirichlet form $\calE^\infty$. Let  $Y$ be  a $T_t^{(k)}$-invariant set for each $k\in\N$. Then
\begin{enumerate}
\item The set $Y$ is $T_t^\infty$-invariant as well.
\item $(\calE^k)^Y$ converges to $(\calE^\infty)^Y$ in the sense of Mosco.
\end{enumerate}
\label{InvarianceMosco}
\end{theo}
\begin{proof}
For a $T_t$-invariant set $Y$, let us consider the  projection $Pu=1_Yu$.\\
To prove assertion (a) it suffices to apply Lemma \ref{projection} with the projection $P$ and to observe that invariance is equivalent to commutability of the semigroup and the projection $P$. To prove  assertion (b), let $v\in L^2(Y,m|_Y)$. Set $\tilde v$ the extension of $v$ by $0$ on $X\setminus Y$. Then  $(T_t^{(k)})^Y v=PT_t^{(k)}P \tilde v\to PT_t^{\infty}P\tilde v=(T_t^{\infty})^Y v$. Once again making use of   \cite[Corollary 2.6.1]{Mosco1994}, we get the claim.
\end{proof}
\begin{rk}
Theorem \ref{InvarianceMosco} indicates that there are much more invariant sets for the limit Dirichlet form then it is for the approximating forms. This may explain why Mosco convergence does not preserve irreducibility. Here are some examples which confirm this observation.
\end{rk}
\begin{exa}{(Decoupling by one $\delta$-interaction)}.
We consider $L^2:=L^2(\R,dx)$,
\[
\dom(\calE^k)=W^{1,2}(\R),\ \calE^k[u]=\int_{\R} (u')^2\,dx + ku^2(0),\ \forall\,u\in\,W^{1,2}(\R),\ k\in\N.
\]
It is easy to check that $(\calE^k)$ is a monotone increasing sequence of irreducible Dirichlet forms. Moreover by Kato's theorem for monotone convergence of closed forms, $(\calE^k)$ converges in the sense of Mosco to the Dirichlet form $\calE^\infty$ defined by
\[
\dom(\calE^\infty)=W_0^{1,2}(\R\setminus\{0\}),\ \calE^\infty[u]=\int_\R (u')^2\,dx,\ \forall\,u\in W_0^{1,2}(\R\setminus\{0\}).
\]
Let us prove that $\calE^\infty$ is however not irreducible. To that end it suffices to prove that $(0,\infty)$ is $T_t^\infty$-invariant. Indeed, let $u\in W_0^{1,2}(\R\setminus\{0\})$. Then $u(0)=0$. Hence the functions $u_1:=1_{(0,\infty)}u\in W_0^{1,2}(\R\setminus\{0\})$ and $u_2:=1_{(-\infty,0)}u\in W_0^{1,2}(\R\setminus\{0\})$. Moreover $\calE^\infty[u] = \calE^\infty[u_1] + \calE^\infty[u_2]$. Hence $\calE^\infty$ is not irreducible. In fact $\calE^\infty$ has two nontrivial invariant sets which are $(-\infty,0)$ and $(0,\infty)$.
\label{OneDelta}
\end{exa}
\begin{exa}{(Decoupling by many $\delta$-interactions)}.
In $L^2:=L^2(\R,dx)$ we consider the family $\calE^n,\ n\in\N$ of Dirichlet forms defined by
\begin{align*}
\dom(\calE^n)&=\{u\in W^{1,2}(\R),\ \sum_{k\in\Z} u^2(k)<\infty\},\\
&\calE^n[u]=\int_\R (u')^2\,dx + n \sum_{k\in\Z} u^2(k),\ \forall\,u\in\,W^{1,2}(\R).
\end{align*}
As $(\calE^n)_{n}$ is monotone increasing, it converges in sense of Mosco to the Dirichlet form $\calE^\infty$, corresponding to the Dirichlet Laplacian on $\R\setminus\Z$. Arguing as before, we see that every interval $(k,k+1),\ k\in\Z$ is $\calE^\infty$-invariant while the $\calE^n$'s are irreducible.
\label{ManyDelta}
\end{exa}
\begin{exa}{(Every measurable set is invariant)}.
In $L^2:=L^2(\R,dx)$ we consider the family $\calE^n,\ n\in\N$ of Dirichlet forms defined by
\begin{align*}
\dom(\calE^n)=W^{1,2}(\R),\ \calE^n[u]=\frac{1}{n}\int_\R (u')^2\,dx + u^2(0),\ \forall\,u\in W^{1,2}(\R).
\end{align*}
As $(\calE^n)_{n}$ is monotone decreasing, it converges in the sense of Mosco to the closable part of the form $u^2(0)$ on $W^{1,2}(\R)$. However, the latter form is not closable. Thus  $\calE^n$ converges to $\calE^\infty =0$ with domain $L^2$, in the sense of Mosco. Hence every Borel measurable subset of $\R$ is $\calE^\infty$-invariant while $\calE^n$ is irreducible for each integer $n$.

\end{exa}

\begin{rk}
{\rm
The mentioned examples extend to higher dimensions if one changes point interactions by $\delta$-sphere interactions.

}
\end{rk}
\section{Decomposition formulae and their consequences}
We turn our attention to establish decomposition formulae for Dirichlet forms. The first one concerns decomposition into transient and recurrent parts.\\
For the concepts of transience and recurrence we shall adopt those introduced by Chen--Fukushima \cite[Chapt. 2]{Chen-Fukushima}. Let us recall them for the convenience of the reader.
Let us consider the family of linear operators as follows:
\[
S_t:L^2\to L^2,\ S_t f=\int_0^t T_s f\,ds,\ t>0.
\]
Then for each $t>0$ the operator $S_t$ is bounded and $\|S_t f\|\leq t\|f\|$ for any $f\in L^2$. By the $\sigma$-finiteness of $(X,m)$ together with the Markov property of $T_t$, the latter inequality extends to elements from $L^2\cap L^1$. This observation enables one to extend operators $S_t$ into bounded operators from  $L^1$ into itself. Moreover   $\|S_t f\|_{L^1}\leq t\|f\|_{L^1}$ for any $f\in L^1$. The same properties hold true for the resolvent family $K_\alpha,\ \alpha>0$. We also quote that $S_t$ and $K_\alpha$ on $L^1$ enjoy both positivity and monotonicity properties: For every $0<s\leq t,\ 0\leq\alpha\leq\beta$ and every $f\in L^1_+$ (the set of positive functions from $L^1$) it holds
\[
0\leq S_s f\leq S_t f\ \text{ and } 0\leq K_\alpha f\leq K_\beta f.
\]
Hence for every $f\in L^1_+$ the function
\[
Kf:=\lim_{n\to\infty} S_n f=\lim_{n\to\infty} K_{1/n} f,
\]
is well defined, with maybe infinite values.
\begin{defi}
We say that the Dirichlet form $\calE$, or the related semigroup $(T_t)_{t>0}$, is transient whenever $Kg<\infty$ for some nonnegative $g\in L^1_+$.\\
The form $\calE$,  or the related semigroup $(T_t)_{t>0}$, is called recurrent whenever   $Kf$ is either $0$ or $\infty$ for any  $f\in L^1_+$.
\end{defi}
Equivalent conditions for transience and recurrent can be found in \cite[Chapt.2]{Chen-Fukushima}. For example, according \cite[Proposition 2.1.3, p.39]{Chen-Fukushima} the semigroup $T_t$ is transient if and only if
\begin{align}
Gf<\infty\ \text{ for some }\  f\in L^1_+.
\label{transience}
\end{align}
While  recurrence of $T_t$ is equivalent to either of the following conditions:
\begin{align}
Gg =\infty\ \text{ for every nonnegative }\ g\in L^1_+,
\label{recurrence1}
\end{align}
or
\begin{align}
 \text{ there is some }\  f\in L^1_+\ \text{ such that }\  Kf=\infty.
\label{recurrence2}
\end{align}
For the readers who are interested to recurrence an transience  in the non-symmetric context we refer to \cite{Beznea}.\\
Let $\rho$ be a fixed  nonnegative function from $L^1_+$ (which exists by the $\sigma$-finiteness of $m$). Then, according to \cite[Lemma 1.6.2, p.54]{Fukushima} the sets
\begin{align*}
X_{\rec}:=\{x\colon\,K\rho(x)=\infty \},\ X_{\trans}:=\{x\colon\,K\rho(x)<\infty \},
\end{align*}
are independent from the choice of the function $\rho$. Moreover they are both invariant.\\
The following result, save uniqueness, was proved by Kuwae in \cite[Theorem 1.3, 1.4]{Kuwae} for non-symmetric quasi-regular semi-Dirichlet forms, in the framework of metric spaces. Let us stress unlike Kuwae, we do not assume quasi-regularity of the considered Dirichlet forms. Furthermore our method deviates from Kuwae's method, especially concerning the recurrent part.
\begin{theo}
Let $\calE$ be a Dirichlet form. Then there are unique quadratic forms $\calE^{\rec}$ and $\calE^{\trans}$ such that
\begin{enumerate}
\item The forms  $\calE^{\rec}$ and $\calE^{\trans}$ are Dirichlet forms in the wide sense in $L^2$.
\item  $\calE^{\rec}$ is recurrent in $L^2(X_{\rec},m|_{X_{\rec}})$ whereas   $\calE^{\trans}$ is transient in $L^2(X_{\trans},m|_{X_{\trans}})$ and
\begin{eqnarray}
\calE[u]=\calE^{\rec}[u] + \calE^{\trans}[u],\ \forall\,u\in\calD.
\end{eqnarray}
\end{enumerate}
\label{Zerlegung1}
\end{theo}
\begin{proof}
Owing to the invariance of the sets $X_{\rec}$ and $X_{\trans}$ together with Lemma \ref{PartOfE}, the quadratic forms $\calE^{X_{rec}}$ and $\calE^{X_{trans}}$ are Dirichlet forms respectively in $L^2(X_{\rec},m|_{X_{\rec}})$ and $L^2(X_{\trans},m|_{X_{\trans}})$. Moreover  $\calE[u]=\calE^{X_{\rec}}[u] + \calE^{X_{\trans}}[u]$ for each $u\in\calD$.\\
Let us prove that $\calE^{\trans}:=\calE^{X_{\trans}}$ is transient as a Dirichlet form in $L^2(X_\trans,m|_{X_\trans})$. Set $K_{\trans}$ the kernel of the Dirichlet form $\calE^{\trans}$. Then the invariance of $X_\trans$ leads to  $K_{\trans} = 1_{X_{\trans}} K|_{L^2( X_{\trans},m|_{X_{\trans}})}$. Let $f=\rho|_{X_{\trans}}$. Then $f\in L^1_+( X_{\trans},m|_{X_{\trans}})$ and
\[
K_{\trans} f = 1_{X_{\trans}} K|_{L^2( X_{\trans},m|_{X_{\trans}})} f = 1_{X_{\trans}}K\rho<\infty.
\]
Hence by (\ref{transience}) $\calE^{\trans}$ is transient in $L^2(X_\trans,m|_{X_\trans})$.\\
Let $\calE^{\rec}:=\calE^{X_{\rec}}$ and let $K_{\rec}$ be the kernel of the Dirichlet form $\calE^{\rec}$. As before the invariance of the set $X_{\rec}$ leads to $K_{\rec} = 1_{X_{\rec}} K|_{L^2( X_{\rec},m|_{X_{\rec}})}$. Thus, setting $f=\rho|_{X_{\rec}}$ we obtain $K_{\rec}f=\infty$. Finally making use of the condition (\ref{recurrence2}) we conclude the recurrence of  $\calE^{\rec}$ in $L^2(X_\rec,m|_{X_\rec})$.\\
Uniqueness follows from Theorem \ref{maximal}.
\end{proof}
\begin{rk}
\begin{enumerate}
\item As a consequence of Theorem \ref{Zerlegung1} we get, $\calE$ is recurrent, respectively transient, if and only if $\calE=\calE^{\rec}$ in the sense that $\calE^\rec$ is a Dirichlet form in $L^2$ and both forms coincide pointwise (equivalently $X=X_{\rec}$-$m$ a.e.), respectively $\calE=\calE^{\trans}$ (equivalently $X=X_{\trans}$-$m$ a.e.).
\item Formula (\ref{Zerlegung1}) is very sensitive to small perturbations. Indeed, let $\calE$ be recurrent. Then $\calE_\epsilon = \calE +\epsilon\|\cdot\|^2,\ \epsilon>0$ is however transient and converges to $\calE$ pointwise (and in the sense of Mosco).
\end{enumerate}
\end{rk}
\begin{prop}
Let $Y\subset X$ be an invariant set. Assume that $\calE$ is transient, resp. recurrent. Then so is $\calE^Y$ in $L^2(Y,m|_Y)$.
\label{PartDirichlet}
\end{prop}
\begin{proof}
The proof is easy, so we omit it.
\end{proof}
The later proposition suggests characterization of the sets $X_{\trans}$ and  $X_{\rec}$ by means of invariant sets in the following way. Let
\begin{align}
{\cal{T}}^{\trans}&:=\{ Y\subset X\colon Y\ \text{is measurable},\  Y \text{ is } T_t-\text{ invariant and }\nonumber\\
& \calE^Y \text{ is } L^2(Y,m|_Y)-\text{ transient }\}
\end{align}
and
\begin{align}
{\cal{R}}^{\rec}&:=\{ Y\subset X\colon Y\ \text{is measurable},\  Y \text{ is } T_t-\text{ invariant and }\nonumber\\
& \calE^Y \text{ is } L^2(Y,m|_Y)-\text{ recurrent }\}.
\end{align}
%
%
%
%
\begin{theo}
The set  $X_{\trans}$, respectively $X_{\rec}$  is the largest element of ${\cal{T}}^{\trans}$, respectively ${\cal{R}}^{\rec}$.
\label{maximal}
\end{theo}
\begin{proof}
We prove only the first part of the theorem, the corresponding conclusion for the recurrent case runs similarly.\\
Let us observe, for any $Y\in{\cal{T}}^{\trans}$ it holds $Y\subset X_{\trans}$. As $X_{\trans}\in {\cal{T}}^{\trans}$, we are done.
\end{proof}
Henceforth, we turn our attention to write any Dirichlet form as the sum of a conservative and a dissipative Dirichlet forms. We shall adopt the standard definitions for the concepts of conservativeness and dissipation. Precisely,  we shall name a Dirichlet form $\calE$ (or its related semigroup $T_t$) dissipative if for some $t>0$ it holds
\begin{align}
m(\{T_t1<1\})>0.
\label{ConservSet}
\end{align}
A Dirichlet form $\calE$ (or its related semigroup $T_t$) is called conservative if
\begin{align}
T_t1 = 1\ \text{ for some and hence every}\ t>0.
\end{align}
The latter is equivalent to
\begin{align}
\alpha K_\alpha 1 = 1\ \text{ for some and hence every}\ \alpha>0.
\end{align}
%
%
We set
\begin{align}
X_{\cons}:=\cap_{t>0}\{T_t1=1\} \text{ and } X_{\diss}:= X\setminus X_{\cons}.
\end{align}
Obviously $\calE$ is dissipative if and only if $m(X_{\diss})>0$, whereas it is conservative if and only if $m(X_{\diss})=0$. Moreover it holds
\begin{align}
X_{\cons}=\cap_{k\in\N}\{T_k1=1\} \text{ and } X_{\diss}:= \cup_{k\in\N} \{T_k1<1\}.
\label{reecriture}
\end{align}
For each $t>0$ let us set
\begin{align*}
E_t:=\{x\colon\,T_t1(x)=1\}.
\end{align*}
A crucial step towards proving the already described decompositions is to prove invariance of the sets $X_\cons,X_\diss$.
\begin{lem}
The sets $X_\cons$ and $X_\diss$ are invariant. Moreover, it holds $X_{\diss}\subset X_{\trans}$ whereas  $X_{\rec}\subset X_{\cons}$.
\label{inclusion}
\end{lem}
\begin{proof}
Owing to formula (\ref{reecriture}) both sets $X_{\cons},X_{\diss}$ are measurable. Let $u\in L^2$. Without loss of generality we may and shall assume that $u$ is positive. Regarding the symmetry of  $T_t$ together with the $\sigma$-finiteness of $m$ and the definitions of $X_{\cons},X_{\diss}$, we obtain  for all $s,t>0$
\begin{align*}
\int_X(1-T_s1)T_t(1_{X_{\cons}}u)\,dm &=\int_{X{_\cons}}(1-T_s1)T_t(1_{X_{\cons}}u)\,dm + \int_{X_{\diss}}(1-T_s1)T_t(1_{X_{\cons}}u)\,dm \\
&= \int_{X_{\diss}}(1-T_s1)T_t(1_{X_{\cons}}u)\,dm\\
&=\int_X T_t((1-T_s1))\cdot 1_{X_{\cons}}u\,dm=\int_{X_{\cons}} (T_t1 - T_{s+t}1)u=0.
\end{align*}
Hence for each fixed $t>0$ we get  $T_t(1_{X_{\cons}}u)=0$ on every set $E_s$. Having the definition of $X_\diss$ in mind we get  $T_t(1_{X_{\cons}}u)=0$ on  $X_{\diss}$. Thus $X_{\cons}$ is invariant and so is $X_{\diss}$.\\
Let us prove the remainder of the lemma. Clearly the second inclusion follows from the first one and we are simply lead to show $X_{\diss}\subset X_{\trans}$. Let $g\in L^1_+$. Changing $g$ by $g\wedge l$ we may and shall assume that $g\in L^1\cap L^\infty$ and hence $g\in L^2\cap L^\infty$. From Lemma \ref{KernelRep}, together with (\ref{monotoneLim}) we obtain
\[
\int_X (1-T_t1)(K_{1/n}g)^2\,dm\leq\calEt[K_{1/n}g]\leq\calE[K_{1/n}g]\leq\int_X g^2\,dm<\infty,\ \forall\,t>0.
\]
Thus for every $t>0$ we have
\[
\int_X (1-T_t1)(K_{1/n}g)^2\,dm =\int_{X_{\diss}} (1-T_t1)(K_{1/n}g)^2\,dm<\infty.
\]
Letting $n\to\infty$, we obtain
\[
\int_{X_{\diss}} (1-T_t1)(Kg)^2\,dm<\infty,\ \forall\,t>0
\]
and hence $Kg<\infty$ on $X_{\diss}$, showing that $X_{\diss}\subset X_{\trans}$.

\end{proof}
\begin{rk}
On the light of Lemma \ref{inclusion} together with the respective  definitions of the sets $X_{\trans},X_{\rec}, X_{\cons}$ and $X_{\diss}$ it holds
\[
X= X_{\rec} \cup (X_{\trans}\setminus X_{\diss}) \cup X_{\diss},
\]
with $a.e.$ disjoint union. This is a refinement of \cite[Theorem 1.3]{Kuwae} in our framework.
\label{SpaceDecomp}
\end{rk}
We are in position now to give the decomposition of a Dirichlet form into a conservative and a dissipative (non-conservative) part. The decomposition is motivated by the fact that there are  much more conservative than recurrent Dirichlet forms and much more transient than dissipative Dirichlet forms.
\begin{theo}{(The second decomposition)}
Let $\calE$ be a Dirichlet form. Then there are unique quadratic forms $\calE^{\cons}$ and $\calE^{\diss}$ with respective domains $\calD$ such that
\begin{enumerate}
\item The forms $\calE^{\cons}$ and $\calE^{\diss}$ are Dirichlet forms in the wide sense in $L^2$.
\item The form $\calE^{\cons}$ is conservative in $L^2(X_{\cons},m|_{X_{\cons}})$ whereas  $\calE^{\diss}$ is dissipative in $L^2(X_{\diss},m|_{X_{\diss}})$ and
\begin{eqnarray}
\calE[u]=\calE^{\cons}[u] + \calE^{\diss}[u],\ \forall\,u\in\calD.
\end{eqnarray}
\end{enumerate}
\label{Zerlegung2}
\end{theo}
\begin{proof}
{\em Existence:} By Lemma \ref{inclusion}  both $X_{\cons}$ and $X_{\diss}$ are $T_t$-invariant subsets. Hence for every  $u\in\calD$ it holds,  $u=u_c+u_d$, where $u_c:=1_{X_{\cons}}u,\ u_d:=1_{X_{\diss}}u$ and $u_{c,d}\in\calD$. Moreover, $\calE[u]=\calE[u_c]+\calE[u_d]$. Set
\begin{align*}
\calE^{\cons}[u]:=\calE[u_c],\ \calE^{\diss}[u]:=\calE[u_d].
\end{align*}
Then both forms are Dirichlet forms in the wide sense in $L^2$ and are Dirichlet forms respectively in $L^2(X_\cons,m|_{X_\cons})$ and $L^2(X_\diss,m|_{X_\diss})$. Moreover  the decomposition holds true. It remains to prove that $\calE^{\cons}$ is conservative in $L^2(X_\cons,m|_{X_\cons})$ whereas $\calE^{\diss}$ is dissipative in $L^2(X_\diss,m|_{X_\diss})$.\\
Let $T_t^{\cons}$ respectively $T_t^{\diss}$ be the semigroups related respectively to $\calE^{\cons}$ and $\calE^{\diss}$ as Dirichlet forms respectively in $L^2(X_\cons,m|_{X_\cons})$ and $L^2(X_\diss,m|_{X_\diss})$. Owing to the invariance of both sets $X_{\cons}, X_{\diss}$ it holds
\begin{eqnarray}
T_t^{\cons} 1_{X_{\cons}}= 1_{X_{\cons}}T_t 1_{X_{\cons}},\ T_t^{\diss}1_{X_{\diss}} =1_{X_{\diss}}T_t 1_{X_{\diss}}.
\label{conssg}
\end{eqnarray}
The latter identities together with the definitions of the sets $X_{\cons}, X_{\diss}$ lead to  $T_t^{\cons}1_{X_{\cons}}= 1_{X_{\cons}}T_t 1=1_{X_{\cons}}$ for every $t>0$. Hence $\calE_{\cons}$ is conservative as a Dirichlet form in $L^2(X_\cons,m|_{X_\cons})$. Besides it holds $T_t^{\diss} 1_{X_{\diss}}= 1_{X_{\diss}}T_t 1<1_{X_{\diss}}$ for some $t>0$. Accordingly $\calE^{\diss}$ is dissipative as a Dirichlet form in $L^2(X_\diss,m|_{X_\diss})$.\\
{\em Uniqueness:}
Assume there is Dirichlet  forms in the wide sense $Q^c, Q^d$ with domain $\calD$ which are respectively conservative and dissipative on some $T_t$- invariant set $E\subset X$ and $X\setminus E$ such that $\calE[u]=Q^c[u] + Q^d[u]$ for every $u\in\calD$. As $X_{\cons}$ is the largest invariant set on which $T_t$ is conservative whereas $X_{\diss}$ is the largest invariant  set on which $T_t$ is dissipative, we get $E\subset X_{\cons}$ and then $ X\setminus E \subset X_{\diss}\subset X\setminus E$. Thus $ X\setminus E = X_{\diss}$ and hence $E=X_{\cons}$. Besides for any $u\in\calD$ it holds $Q^c[1_{X_c}u]= Q^c[u]=\calE^{\cons}[u],\ Q^d[u]=Q^d[1_{X_d}u]=\calE^{\diss}[u]$ from which follows $\calE^{\cons}=Q^c$ and $\calE^d=Q^{\diss}$.
\end{proof}
\begin{rk}
\begin{enumerate}
\item We shall call $\calE^{\cons}$ the conservative part of $\calE$, while $\calE^{\diss}$ is the dissipative part of $\calE$. Besides, we shall name $X_\cons$, respectively $X_\diss$, the conservative, respectively the dissipative, space of $\calE$.   Let us emphasize that our connotations for conservative and dissipative spaces differ  from those introduced in \cite[p.55]{Fukushima} or \cite{Kuwae}.
\item Assume that $X$ is a locally compact separable metric space and $\calE$ is quasi-regular. In  this case one can associate to $\calE$ a right continuous Markov process. Moreover a.e. properties can be replaced by quasi-everywhere notions. Hence Theorem \ref{Zerlegung1} indicates the possibility of decomposing q.e. the process of $\calE$ into the sum of transient and a recurrent process. Whereas Theorem \ref{Zerlegung2} indicates that the process related to $\calE$ decomposes q.e. into the sum of a process with an infinite lifetime and an other one with a finite lifetime.
\end{enumerate}
\end{rk}
Let us now analyze the relationship between the respective parts of a Dirichlet form.
\begin{prop}
\begin{enumerate}
\item The form $\calE$ is conservative, respectively dissipative if and only if $\calE=\calE^\cons$ (or equivalently $X=X_\cons$ a.e.),  respectively, if and only if $\calE=\calE^\diss$ (or equivalently $X=X_\diss$ a.e.).
\item Every recurrent Dirichlet form is conservative, whereas  every dissipative Dirichlet form is transient.
\item If either $m(X_\cons)<\infty$ or $m(X_\trans)<\infty$ then $\calE^\cons=\calE^\rec$ and $\calE^\diss=\calE^\trans$.
\end{enumerate}
\label{PropCons}
\end{prop}
\begin{proof}
The proof of the first assertion is obvious, so we omit it.\\
(b)-(c): According to Lemma \ref{inclusion}, we have $X_\rec\subset X_\cons$ and $X_\diss\subset X_\trans$. Thus if $\calE$ is recurrent, respectively dissipative we  obtain $X=X_\rec=X_\cons$ and then $\calE$ is conservative, respectively $X=X_\diss=X_\trans$ and then $\calE$ is transient.\\
(c): It is well known that in case $m(X)<\infty$ then recurrence and conservativeness coincide. Thus if $m(X_\cons)<\infty$ then $\calE_\cons$ is recurrent. This leads to $X_\rec=X_\cons$ and hence $X_\diss=X_\trans$, from which the assertion follows. Assume now that $m(X_\trans)<\infty$. According to the first part of the proof,  the conservative part of $\calE^{\trans}$ is recurrent and hence vanishes. Consequently, $\calE^{\trans}=\calE^{\diss}$ yielding the equality $\calE^{\rec}=\calE^{\cons}$.

\end{proof}
\begin{rk}
It may happen that $\calE^\rec\neq\calE^\cons$ or $\calE^\trans\neq\calE^\diss$. For, take the Dirichlet form $\calE$ related to the gradient energy form in $L^2(\R^d,dx)$ with $d\geq 3$, i.e.
\begin{align*}
\calD:= H^1(\R^d),\ \calE[u]=\int_{\R^d}|\nabla u|^2\,dx,\ \forall\,u\in\calD.
\end{align*}
It is well known that $\calE$ is transient and conservative. Thus $\calE^\rec=0=\calE^\diss$ whereas $\calE^\cons =\calE^\trans=\calE$.
\end{rk}
\begin{theo}{(The synthesis)}
Every Dirichlet form decomposes into the sum of a recurrent, dissipative and transient-conservative Dirichlet forms
\label{Zerlegung3}
\end{theo}
\begin{proof}
According to Lemma \ref{inclusion} $X_{\diss}$ is $T_t$-invariant. Hence the set $X_{tc}:=X_{\trans}\setminus X_{\diss}$ is $T_t$-invariant as well. Regarding the inclusion $ X_{\diss}\subset X_{\trans}$ we derive
\[
\calE^{\trans}[u] = \calE^{\diss}[u] + \calE^{X_{tc}} [u],\ \forall\,u\in\calD.
\]
From the very definition of $X_{tc}$ we learn that $\calE^{X_{tc}}$ is transient-recurrent in $L^2(X_{tc},m|_{X_{tc}})$. Applying Theorem \ref{Zerlegung1}, we get the result.
\end{proof}

Relying on Theorem \ref{Zerlegung3} we immediately derive the following.
\begin{coro}
Assume that $T_t$ is irreducible. Then either $\calE$ is recurrent or dissipative or transient-conservative.
\end{coro}
\begin{prop}
Assume that $\calE$ is conservative. If $m(X_{\trans})>0$ then $\calE^{\trans}$ is conservative as well.
\end{prop}
\begin{proof}
If $m(X_{\rec})=0$, then $\calE=\calE^{\trans}$ and we are done. If not, the form $\calE^{\rec}$ is nonzero  and is conservative in $L^2(X_\rec,m|_{X_\rec})$. Thus
\begin{align}
T_t1 = 1_{X_\rec}T_t 1_{X_\rec} + 1_{X_\trans}T_t 1_{X_\trans}=1_{X_\rec}+ 1_{X_\trans}T_t 1_{X_\trans}=1,
\end{align}
leading to $1_{X_\trans}T_t 1_{X_\trans}=1_{X_\trans}$. Hence $\calE^{\trans}$ is conservative.
\end{proof}
Yet we turn our attention  to analyze relationship between conservativeness of $\calE$ and its approximating sequences $\calEb$ and $\calEt$.
\begin{theo}
For each $\beta,t>0$, set   $Y_{\cons}^{(\beta)}, Y_{\diss}^{(\beta)},X_{\cons}^{(t)}, X_{\diss}^{(t)}$ respectively the conservative and the dissipative spaces of $\calEb$ and $\calEt$. Then
\begin{enumerate}
\item For every $\beta,t>0$, it holds $X_\cons= X_{\cons}^{(t)}=Y_{\cons}^{(\beta)}$ and hence  $X_\diss= X_\diss^{(t)}=Y_{\diss}^{(\beta)}$.
\item  $\calE$ is conservative if and only if $\calEt$ also is for some and hence every $t>0$, equivalently $\calEb$ is conservative for some and hence every $\beta>0$ . Analogous statement holds true for dissipation.
\end{enumerate}
\end{theo}
\begin{proof}
Obviously assertion (b) is a direct consequence of assertion (a).\\
$X_\cons\subset X_{\cons}^{(t)}$: Following the extension of $T_t$ on $L^\infty$, we get that the extension of  $T_s^{(t)}$ from $L^2\cap L^\infty$ to $L^\infty$ is given via
\begin{align}
T_s^{(t)}u=e^{-s/t}\sum_{k=0}^\infty \frac{1}{k!}(s/t)^k T_t^k u , \forall\,s,t>0,\ u\in L^\infty.
\end{align}
Hence, as $X_{\cons}$ is $\calEt$-invariant we obtain
\[
1_{X_{\cons}}T_s^{(t)} 1= T_s^{(t)} 1_{X_{\cons}}=\exp(-s/t(1-T_t))1_{X_{\cons}} =e^{-s/t}\sum_{k=0}^\infty \frac{1}{k!}(s/t)^k T_t^k 1_{X_{\cons}} , \forall\,s,t>0.
\]
Induction on $k$ leads to
\[
T_t^k( 1_{X_{\cons}}  )= T_t^{k-1}(1_{X_{\cons}})=1_{X_{\cons}}T_t^k 1=1_{X_{\cons}},\ \forall\,t>0,\ k\in\N.
\]
Thus $1_{X_{\cons}}T_s^{(t)} 1 =1_{X_{\cons}}$ leading to the inclusion $X_\cons\subset X_{\cons}^{(t)}$.\\
$X_{\cons}^{(t)}\subset X_\cons$: Since $\calEt_\cons$ is conservative as a Dirichlet form in $L^2(X_\cons,m|_{X_\cons})$, then the lower bound of its the spectrum is $0$. Hence from the  representation of $\calEt$ (see Lemma \ref{KernelRep}), we infer that
\[
\int_{X_\diss\cap X_{\cons}^{(t)}} u^2\,dm=0,\ \forall\,u\in L^2.
\]
Thus $ m( X_\diss\cap X_{\cons}^{(t)})=0$ and  $X_{\cons}^{(t)}\subset X_\cons$. Now both inclusions lead to $X_{\cons}^{(t)}=X_\cons$.\\
The proof of the claim $X_\cons=Y_\cons^{(\beta)}$ follows exactly the lines of the preceding one, so we omit it.

\end{proof}
Combining the latter theorem with Theorem \ref{InvarianceMosco}-(b) we obtain:
\begin{coro}
\begin{enumerate}
\item The conservative part of $\calE$ is the Mosco limit of the conservative part of $\calEt$ whereas its dissipative part is the Mosco limit of the dissipative part of $\calEt$.
\item The conservative part of $\calE$ is the Mosco limit of the conservative part of $\calEb$ whereas its dissipative part is the Mosco limit of the dissipative part of $\calEb$.
\end{enumerate}
\end{coro}
\begin{rk}{\rm
Let us give a final remark concerning the solution of the heat equation. Let $L$ be the positive selfadjoint operator related to $\calE$ and $u\in L^2$. Then the solution of the heat equation
\begin{eqnarray}
\label{heat}
\left\{\begin{gathered}
-\frac{\partial v}{\partial t}=Lv\\
v(0,\cdot)= u
\end{gathered}
\right.
\end{eqnarray}
is given by $u(t):=T_tu,\ t\geq 0$ (set $T_0=1$).\\
We recall that a positive function $u\in L^2$ is called excessive if $T_tu\leq u$. It is well known that $T_t$ is conservative if and only if for every excessive function $u$ it holds $T_tu=u$. Thus For excessive initial data $u$ we have
\[
u(t)=T_tu= 1_{X_\cons}u + 1_{X_\diss}T_tu.
\]
In other words, for excessive initial date, the solution of the heat equation is the sum of an autonomous function and a time dependent one. Moreover, though the $L^2$-norm of the solution $u(t)$ decreases, the latter identity shows that $\|u(t)\|_{L^2}$ is lower semi-bounded by $\|u\|_{L^2(X_\cons)}$. More strongly, if $m(X_\cons)>0$ and $u$ is non negative then
\[
\liminf_{t\to\infty} u(t)\geq 1_{X_\cons}u>0.
\]

}
\end{rk}

\section{Examples}
\subsection{Parts of a one-dimensional diffusion on intervals}
Consider $I_1:=[-2,-1]$, $I_2:=[0,\infty)$, $X = I_1 \cup I_2$. Let  $m$ be the measure with full support $X$ defined by $d\,m(x)=2x^2 {\mathrm d}x$  and $s$ be the function
\[
s:I_1\cup (0,\infty)\to\R,\ x\mapsto -\frac{1}{x}
\]
Let us denote by $AC_s(X)$ the space of $s$-absolutely continuous functions on $X$. Set
\[
\calD_0:=\Bigl\{u\from X\to\R:\; u\in AC_s(X),\ \int_X (u'(x))^2 x^2\,{\mathrm d}x<\infty\Bigr\},
\]
We define  $\calE$ the Dirichlet form in $L^2(X, m)$ by
\begin{align*}
  \calD :=\calD_0\cap L^2(X,m),\ \calE[u]:= \int_X (u'(x))^2 x^2\,{\mathrm d}x \quad\text{for all }u\in \calD.
\end{align*}
Let us prove that $I_1$ is an invariant set w.r.t. $\calE$. To that end let $u\in\calD$. Obviously ${1}_{I_1}u, {1}_{I_2}u\in\calD$ and $\calE[u] = \calE[{1}_{I_1}u] + \calE[{1}_{I_2}u]$. Hence $I_1$ is invariant and so is $I_2$.
Let $\calE^{I_i},\ i=1,2$ be the forms defined by t
\[
 \calD^{I_2} := dom(\calE^{I_2})= \calD, \ \ \calE^{I_2}[u] = \calE[{1}_{I_2}u],
 \]
and
\[
\calD^{I_1} := dom(\calE^{I_1})= \calD, \ \ \calE^{I_1}[u] = \calE[{1}_{I_1}u].
\]
Both forms are Dirichlet forms in the wide sense and are in fact Dirichlet forms on the respective spaces $L^2(I_i,m|_{I_i})$. Moreover $\calE$ decomposes into the sum of both forms. Let us now characterize the parts of $\calE$.\\
We claim that $\calE^{I_1}$ is a recurrent Dirichlet form in $L^2 (I_1, m|I_1)$. Indeed, both endpoints $-2$ and $-1$ are regular, i.e. $s(-2) = \frac{1}{2}< \infty $ , $s(-1) = 1 < \infty$ and $m((-2,c))<\infty $,  $m((c,-1))<\infty $ for all $c\in I_1$. Thereby and according to Feller's classification of one-dimensional diffusions (see  \cite[Prop.2.2.8, p.66]{Chen-Fukushima}) we conclude that $\calE^{I_1}$ is recurrent in $L^2 (I_1, m|I_1)$. Thus it is the recurrent part of $\calE$.\\
Now we claim that $\calE^{I_2}$ is  transient and conservative in $L^2(I_2, m|_{I_2})$. Indeed, owing to the fact that $0$ is non-approachable, i.e. $s(0) = -\infty$ while $\infty$ is approachable and a non-regular endpoint, i.e. $s(\infty) = 0<\infty$  and $m((c,\infty))= \infty$ for all $c>0$. Hence, owing to \cite[Proposition 2.2.11, p.68]{Chen-Fukushima}, the form $\calE^{I_2}$ is transient. On the other hand an elementary computation yields
\[
\int_{0}^c m((x,c))\,ds(x) = \int_c^\infty m((c,x))\,ds(x) =\infty, \forall c>0.
\]
Thereby $\calE^{I_2}$ fulfills Feller's test of non-explosion (see \cite[p. 126]{Chen-Fukushima} and then it is conservative. Summarizing, we obtain  $\calE^{\rec} = \calE^{I_1}$ while $ \calE^{\trans} = \calE^{I_2}$ and $\calE^{\cons} = \calE$.

\subsection{Parts of the trace of the Bessel process}
In this example we choose the function $s$ as before. Let $I:= [0,1]$ , $J:= (a_k)_{k\geq 2}$ with $a_k>0$ for all integers $k$, and $X:= I\cup J$. Let us consider a discrete measure $\mu:= \sum_{k\geq 2} a_k \delta_k$ and the measure $m:= {1}_{I} x^2 {\mathrm d}x + \mu $. We assume that $\mu$ is infinite. We consider the trace of the Bessel process with respect to the measure $m$ (see \cite{BM}) defined by
\begin{align*}
& \check{\calD} := \{ u\in L^2(X, m) : u\in AC_s([0,1]),  \int_0^1 (u'(x))^2 x^2 dx + \sum_{k=2}^\infty k(k+1) (u_{k+1} - u_k)^2 <\infty \}\\
&\ccalE[u] := \int_0^1 (u'(x))^2 x^2 dx + \sum_{k=2}^\infty k(k+1) (u_{k+1} - u_k)^2, \ \text{for all }\ u\in \check{\calD}.
\end{align*}
Following the proof of \cite[Theorem 3.6]{BM}  one can show that $\ccalE$ is transient (in fact it is the trace of a transient Dirichlet form and hence it is transient). Thus $\ccalE^{\trans}=\ccalE$.\\
It is not hard to realize that the sets $I,J$ are invariant. We define the Dirichlet forms $\ccalE^{I}$ and $\ccalE^{J}$ in $L^2(I, m|_{I})$ and $L^2(J, m|_{J})$ respectively by
\[
\dom(\ccalE^{I}) = \ccalD,\ \ccalE^{I}[u] =  \calE[{1}_{ I}u]\ \text{and}\
\dom(\ccalE^{J}) = \ccalD, \ \ccalE^{J}[u] =  \calE[{1}_{ J}u].
\]
Obviously $\ccalE[u] = \ccalE[{1}_{I}u] + \ccalE[{1}_{J}u]$ for each $u\in\ccalD$. On the one hand  a straightforward computation leads to
\[
\int_0^c m((x,c)) {\mathrm d}s(x) =\infty , \ \text{and}\ \int_c^1 m((c,x)) {\mathrm d}s(x) < \infty, \ \forall c\in I.
\]
Hence according to Feller's test once again, we conclude that $\ccalE^{I}$ is dissipative. On the other one, according to \cite[Theorem 3.7]{BM} the discrete part of the form, namely $\ccalE^J$ is conservative if and only if
\begin{eqnarray}
\sum_{k=2}^\infty \frac{a_k}{k} = \infty.
\label{conservDiscrte}
\end{eqnarray}
In conclusion if condition (\ref{conservDiscrte}) is fulfilled then $\ccalE^\diss=\ccalE^I$ while  $\ccalE^\cons=\ccalE^J$. However, if (\ref{conservDiscrte}) is not fulfilled then $\ccalE=\ccalE^\diss=\ccalE^\trans$.

\bibliography{BiblioDecomposition}

\end{document}